\documentclass[11pt,twoside]{amsart}

\usepackage{amssymb}
\usepackage{latexsym}

\catcode`\@=11
\def\omathop#1#2#3{\let\temp=#1\def\letter{#2}
  \ifcat#3_ \let\next\@@olim\else\let\next\@olim\fi\next#3}
\def\@olim{\letter\text{-}\!\temp}
\def\@@olim_#1{\mathchoice{
   \setbox0=\hbox{$\displaystyle\letter\text{-}\!\temp\!\text{-}\letter$}
   \setbox2=\hbox{$\displaystyle\temp$}
   \setbox4=\hbox{$\scriptstyle#1$}
   \dimen@=\wd4 \advance\dimen@ by -\wd2 \divide\dimen@ by2
   \def\next{\letter\text{-}\!\temp_{\hbox to 0pt{\hss$\scriptstyle#1$\hss}}
     \hskip\dimen@}
   \ifdim\wd2>\wd4 \def\next{\@olim_{#1}}\fi
   \ifdim\wd4>\wd0 \def\next{\mathop{\llap{$\letter$-}\!\temp}\limits_{#1}}\fi
   \next}
   {\@olim_{#1}}{\@olim_{#1}}{\@olim_{#1}}}

\def\bosum{\omathop{\sum}{bo}}
\def\olim{\omathop{\lim}{o}}
\def\bolim{\omathop{\lim}{bo}}

\catcode`\@=13

\newcommand{\reduce}{\mskip-2mu}
\newcommand{\ls}{\reduce\left\bracevert\reduce\vphantom{X}}
\newcommand{\rs}{\reduce\vphantom{X}\reduce\right\bracevert\reduce}

\usepackage{color}

\theoremstyle{plain}
\newtheorem{thm}{Theorem}[section]

\newtheorem{cor}[thm]{Corollary}
\newtheorem{lemma}[thm]{Lemma}

\theoremstyle{definition}
\newtheorem{definition}[thm]{Definition}
\newtheorem{example}{Example}

\numberwithin{equation}{section}

\begin{document}

\title{Dominated   Uryson operators on lattice-normed spaces}

\author{Marat Pliev}

\address{South Mathematical Institute of the Russian Academy of Sciences\\
str. Markusa 22,
Vladikavkaz, 362027 Russia}

%\email{maratpliev@gmail.com}

\author{Batradz~Tasoev}

\address{South Mathematical Institute of the Russian Academy of Sciences\\
str. Markusa 22,
Vladikavkaz, 362027 Russia}

%\email{tasoevbatradz@yandex.ru}

%\centerline{\today}



\begin{abstract}
The aim of this note is to introduce the space $\mathcal{D}_{U}(V,W)$ of the dominated Uryson operators on lattice-normed spaces.  We prove an ``Up-and-down''  type theorem for a positive abstract Uryson operator defined on a vector lattice and  taking values in a Dedekind complete vector lattice. This result we apply to prove the decomposability of the lattice valued norm of the space $\mathcal{D}_{U}(V,W)$ of all dominated Uryson operators. We obtain that, for a lattice-normed space $V$ and a Banach-Kantorovich space $W$ the space $\mathcal{D}_{U}(V,W)$ is also a Banach-Kantorovich space. We prove that under some mild conditions, a dominated Uryson operator has an exact dominant. We obtain formulas for calculating the exact dominant of a dominated Uryson operator. We also consider laterally continuous and completely additive dominated Uryson operators and prove that a dominated Uryson operator is laterally continuous (completely additive) if and only if so is its exact dominant.
\end{abstract}

\maketitle

%\tableofcontents

\section{Introduction}

Today the theory of regular operators in vector lattices is a very large area of Functional Analysis to which many textbooks are devoted {\cite{Ab,Al,Ku,Za}}. Nonlinear maps between vector lattices in an involved subject. An interesting class of nonlinear maps called abstract Uryson  operators   was introduced and studied in 1990 by Maz\'{o}n and  de Le\'{o}n \cite{Maz-1,Maz-2,Seg}, and then considered to be defined on lattice-normed spaces by Kusraev and the first named author \cite{Ku-1,Ku-2,Pl-3}. The space of all abstract Uryson operators has nice order properties, but the structure of this space is still less known.
The aim of this note is to investigate the space of dominated Uryson operators. This class of operators generalizes abstract Uryson operators considered by Maz\'{o}n and de Le\'{o}n. We study the set of the fragments of a positive abstract Uryson operator,  prove an ``Up-and-down''  type theorem and  then apply the result  to prove a decomposability of the space of the dominated Uryson operators acted between lattice-normed spaces. We also prove that dominated Uryson operator is laterally continuous (completely additive) if and only if the exact dominant is.

\section{Preliminary information}

The  goal of this section is to introduce some basic definitions and facts. General information on vector lattices and lattice-normed spaces the reader can find in the books \cite{Al,Ku,Za}.

\begin{definition}\label{lat-norm}
Consider a vector space $V$ and a real  archimedean vector lattice
$E$. A map $\ls \cdot\rs:V\rightarrow E$ is a \textit{vector norm} if it satisfies the following axioms:
\begin{enumerate}
  \item[1)] $\ls v \rs\geq 0;$\,\, $\ls v\rs=0\Leftrightarrow v=0$;\,\,$(v\in V)$;
  \item[2)] $\ls v_1+v_2 \rs\leq \ls v_1\rs+\ls v_2 \rs;\,\, ( v_1,v_2\in V)$;
  \item[3)] $\ls\lambda  v\rs=|\lambda|\ls v\rs;\,\, (\lambda\in\Bbb{R},\,v\in V)$.
\end{enumerate}
A vector norm is is said to be \textit{decomposable} if
\begin{enumerate}
  \item[4)] for all $e_{1},e_{2}\in E_{+}$ and $x\in V$ the condition $\ls x\rs=e_{1}+e_{2}$ implies the existence of $x_{1},x_{2}\in V$ such that $x=x_{1}+x_{2}$ and $\ls x_{k}\rs=e_{k}$, $(k:=1,2)$.
\end{enumerate}
In the case where condition $(4)$ is valid only for disjoint $e_{1}, e_{2}\in E_{+}$, the norm is said to be {\it disjointly-decomposable} or, in short, {\it d-decomposable}.
\end{definition}
A triple $(V,\ls\cdot\rs,E)$ (in brief $(V,E),(V,\ls\cdot\rs)$ or $V$ with default parameters omitted) is a \textit{lattice-normed space} if $\ls\cdot\rs$ is an $E$-valued vector norm in the vector space $V$. If the norm $\ls\cdot\rs$ is decomposable then the space $V$ itself is called decomposable. We say that a net $(v_{\alpha})_{\alpha\in\Delta}$ {\it $(bo)$-converges} to an element $v\in V$ and write $v=\bolim v_{\alpha}$ if there exists a decreasing net $(e_{\gamma})_{\gamma\in\Gamma}$ in $E$ such that $\inf_{\gamma\in\Gamma}(e_{\gamma})=0$ and for every $\gamma\in\Gamma$ there is an index $\alpha(\gamma)\in\Delta$ such that $\ls v-v_{\alpha(\gamma)}\rs\leq e_{\gamma}$ for all $\alpha\geq\alpha(\gamma)$. A net $(v_{\alpha})_{\alpha\in\Delta}$ is called \textit{$(bo)$-fundamental} if the net $(v_{\alpha}-v_{\beta})_{(\alpha,\beta)\in\Delta\times\Delta}$ $(bo)$-converges to zero. A lattice-normed space is called {\it $(bo)$-complete} if every $(bo)$-fundamental net $(bo)$-converges to an element of this space. Let $e$ be a positive element of a vector lattice $E$.
By $[0,e]$ we denote the set $\{v\in V:\,\ls v\rs\leq e\}$.
A set $M\subset V$ is called  $\text{(bo)}$-{\it bounded } if there exists
$e\in E_{+}$ such that  $M\subset[0,e]$. Every decomposable $(bo)$-complete lattice-normed space is called a {\it Banach-Kantorovich space} (a BKS for short).

Let $(V,E)$ be a lattice-normed space.  A subspace $V_{0}$ of $V$ is called a $\text{(bo)}$-ideal of $V$ if for any $v\in V$ and $u\in V_{0}$, from $\ls v\rs\leq\ls u\rs$ it follows that $v\in V_{0}$. A subspace $V_{0}$ of a decomposable lattice-normed space $V$  is a $\text{(bo)}$-ideal if and only if $V_{0}=\{v\in V:\,\ls v\rs\in L\}$, where $L$ is an order ideal in $E$ [15,\,2.1.6.1]. Let $V$ be a lattice-normed space and $y,x\in V$. If $\ls x\rs\bot\ls y\rs=0$ then we call the elements $x,y$ {\it disjoint} and write $x\bot y$. The equality $x=\coprod_{i=1}^{n}x_{i}$ means that $x=\sum\limits_{i=1}^{n}x_{i}$ and $x_{i}\bot x_{j}$ if $i\neq j$. An element $z\in V$ is called a {\it component} or a \textit{fragment} of $x\in V$ if  $z\bot(x-z)$. The set of all  fragments of an element $x\in V$ is denoted by $\mathcal{F}_{x}$. The notations $z\sqsubseteq x$ means that $z$ is a fragment of $x$. The Boolean algebra of projections in $V$ is denoted by $\mathfrak{B}(V)$. Observe that if $E=\ls V\rs^{\bot\bot}$ is a vector lattice with the projection property and $V$ is decomposable then the Boolean algebras $\mathfrak{B}(V)$ and $\mathfrak{B}(E)$ are isomorphic.

Consider some important examples of lattice-normed spaces.
\begin{example}
We begin with simple extreme cases, namely vector lattices and normed spaces. If $V=E$ then the modules of an element can be taken as its lattice norm: $\ls v\rs:=|v|=v\vee(-v);\,v\in E$. Decomposability of this norm easily follows from the Riesz Decomposition Property holding in every vector lattice. If $E=\Bbb{R}$ then $V$ is a normed space.
\end{example}
\begin{example}
Let $Q$ be a compact  and let $X$ be a Banach space. Let $V:=C(Q,X)$ be the space of  continuous vector-valued functions from $Q$ to $X$. Assign $E:=C(Q,\Bbb{R})$. Given $f\in V$, we define its lattice norm by the relation $\ls f\rs:t\mapsto\|f(t)\|_{X}\,(t\in Q)$. Then $\ls\cdot\rs$ is a decomposable norm (\cite{Ku},\,Lemma~2.3.2).
\end{example}
\begin{example}
Let $(\Omega,\Sigma,\mu)$ be a $\sigma$-finite measure space,
let $E$ be an order-dense ideal in  $L_{0}(\Omega)$ and let $X$ be a Banach space.
By $L_{0}(\Omega,X)$ we denote the space of (equivalence classes of)  Bochner $\mu$-measurable vector functions
acting from $\Omega$ to $X$. As usual, vector-functions are equivalent if they have
equal values at almost all points of the set $\Omega$. If $\widetilde{f}$ is the coset of a measurable vector-function $f:\Omega\rightarrow X$ then $t\mapsto\|f(t)\|$,$(t\in\Omega)$ is a scalar measurable function whose coset is denoted by the symbol $\ls\widetilde{f}\rs\in L_{0}(\mu)$. Assign by definition
$$
E(X):=\{f\in L_{0}(\mu,X):\,\ls f\rs\in E\}.
$$
Then $(E(X),E)$ is a lattice-normed space with a decomposable norm (\cite{Ku},\,Lemma~2.3.7). If $E$ is a Banach lattice then the lattice-normed space $E(X)$ is  a Banach space with respect to the norm $|\|f|\|:=\|\|f(\cdot)\|_{X}\|_{E}$.
\end{example}

\begin{definition} \label{def:ddmjf0}
Let $E$ be a vector lattice, and let $F$ be a real linear space. An operator $T:E\rightarrow F$ is called \textit{orthogonally additive} if $T(x+y)=T(x)+T(y)$ whenever $x,y\in E$ are disjoint.
\end{definition}

It follows from the definition that $T(0)=0$. It is immediate that the set of all orthogonally additive operators is a real vector space with respect to the natural linear operations.

\begin{definition}
Let $E$ and $F$ be vector lattices. An orthogonally additive operator $T:E\rightarrow F$ is called:
\begin{itemize}
  \item \textit{positive} if $Tx \geq 0$ holds in $F$ for all $x \in E$;
  \item \textit{order bounded} if $T$ maps order bounded sets in $E$ to order bounded sets in $F$.
\end{itemize}
An orthogonally additive, order bounded operator $T:E\rightarrow F$ is called an \textit{abstract Uryson} operator.
\end{definition}
For example, any linear  operator $T\in L_{+}(E,F)$ defines a positive abstract Uryson operator by $G (f) = T |f|$ for each $f \in E$.
Observe that if $T: E \to F$ is a positive orthogonally additive operator and $x \in E$ is such that $T(x) \neq 0$ then $T(-x) \neq - T(x)$, because otherwise both $T(x) \geq 0$ and $T(-x) \geq 0$ imply $T(x) = 0$. So, the above notion of positivity is far from the usual positivity of a linear operator: the only linear operator which is positive in the above sense is zero. A positive orthogonally additive operator need not be order bounded. Consider, for example, the real function $T: \mathbb R \to \mathbb R$ defined by
$$
T(x)=\begin{cases} \frac{1}{x^{2}}\,\,\,\text{if $x\neq
0$}\\
0\,\,\,\,\,\text{if $x=0$ }.\\
\end{cases}
$$

The set of all abstract Uryson operators from $E$ to $F$ we denote by $\mathcal{U}(E,F)$. Consider some examples.
The most famous one is the nonlinear integral Uryson operator.

\begin{example}\label{Ex-0}
Let $(A,\Sigma,\mu)$ and $(B,\Xi,\nu)$ be $\sigma$-finite complete measure spaces, and let $(A\times B,\mu\times\nu)$ denote the completion of their product measure space. Let $K:A\times B\times\Bbb{R}\rightarrow\Bbb{R}$ be a function satisfying the following conditions\footnote{$(C_{1})$ and $(C_{2})$ are called the Carath\'{e}odory conditions}:
\begin{enumerate}
  \item[$(C_{0})$] $K(s,t,0)=0$ for $\mu\times\nu$-almost all $(s,t)\in A\times B$;
  \item[$(C_{1})$] $K(\cdot,\cdot,r)$ is $\mu\times\nu$-measurable for all $r\in\Bbb{R}$;
  \item[$(C_{2})$] $K(s,t,\cdot)$ is continuous on $\Bbb{R}$ for $\mu\times\nu$-almost all $(s,t)\in A\times B$.
\end{enumerate}
Given $f\in L_{0}(A,\Sigma,\mu)$, the function $|K(s,\cdot,f(\cdot))|$ is $\mu$-measurable  for $\nu$-almost all $s\in B$ and $h_{f}(s):=\int_{A}|K(s,t,f(t))|\,d\mu(t)$ is a well defined and $\nu$-measurable function. Since the function $h_{f}$ can be infinite on a set of positive measure, we define
$$
\text{Dom}_{A}(K):=\{f\in L_{0}(\mu):\,h_{f}\in L_{0}(\nu)\}.
$$
Then we define an operator $T:\text{Dom}_{A}(K)\rightarrow L_{0}(\nu)$ by setting
$$
(Tf)(s):=\int_{A}K(s,t,f(t))\,d\mu(t)\,\,\,\,\nu-\text{a.e.}\,\,\,\,(\star)
$$
Let $E$ and $F$ be order ideals in $L_{0}(\mu)$ and $L_{0}(\nu)$ respectively, $K$ a function satisfying $(C_{0})$-$(C_{2})$. Then $(\star)$ defines an \textit{orthogonally additive order bounded integral operator} acting from $E$ to $F$ if $E\subseteq \text{Dom}_{A}(K)$ and $T(E)\subseteq F$.
\end{example}

\begin{example} \label{Ex-1}
We consider the vector space $\mathbb R^m$, $m \in \mathbb N$ as a vector lattice with the coordinate-wise order: for any $x,y \in \mathbb R^m$ we set $x \leq y$ provided $e_i^*(x) \leq e_i^*(y)$ for all $i = 1, \ldots, m$, where $(e_i^*)_{i=1}^m$ are the coordinate functionals on $\mathbb R^m$. Let $T:\Bbb{R}^{n}\rightarrow\Bbb{R}^{m}$. Then $T\in\mathcal{U}(\Bbb{R}^{n},\Bbb{R}^{m})$ if and only if there are real functions $T_{i,j}:\Bbb{R}\rightarrow\Bbb{R}$,
$1\leq i\leq m$, $1\leq j\leq n$ satisfying $T_{i,j}(0)=0$ such that
$$
e_i^*\bigl(T(x_{1},\dots,x_{n})\bigr) = \sum_{j=1}^{n}T_{i,j}(x_{j}),
$$
In this case we write $T=(T_{i,j})$.
\end{example}

\begin{example} \label{Ex2}
Let $T:l^{2}\rightarrow\Bbb{R}$ be the operator defined by
$$
T(x_{1},\dots,x_{n},\dots) = \sum_{n\in I_{x}}n\bigl(|x_{n}|-1\bigr)
$$
where $I_{x}:=\{n\in\Bbb{N}:\,|x_{n}|\geq 1\}$. It is not difficult to check that $T$ is a positive abstract Uryson operator.
\end{example}

\begin{example} \label{Ex3}
Let $(\Omega, \Sigma, \mu)$ be a measure space, $E$ a sublattice of the vector lattice $L_0(\mu)$ of all equivalence classes of $\Sigma$-measurable functions $x: \Omega \to \mathbb R$, $F$ a vector lattice and $\nu: \Sigma \to F$ a finitely additive measure. Then the map $T: E \to F$ given by $T(x) = \nu({\rm supp} \, x)$ for any $x \in E$, is an abstract Uryson operator which is positive if and only if $\nu$ is positive.
\end{example}

Consider the following order in $\mathcal{U}(E,F):S\leq T$ whenever $T-S$ is a positive operator. Then $\mathcal{U}(E,F)$
becomes an ordered vector space.  If a vector lattice $F$ is Dedekind complete we have the following theorem.
\begin{thm}(\cite{Maz-1},Theorem~3.2)\label{th-1}.
Let $E$ and $F$ be a vector lattices, $F$ Dedekind complete. Then $\mathcal{U}(E,F)$ is a Dedekind complete vector lattice. Moreover for $S,T\in \mathcal{U}(E,F)$ and for $f\in E$ following hold
\begin{enumerate}
\item~$(T\vee S)(f):=\sup\{Tg+Sh:\,f=g+h;\,g\bot h\}$.
\item~$(T\wedge S)(f):=\inf\{Tg+Sh:\,f=g+h;\,g\bot h\}.$
\item~$(T)^{+}(f):=\sup\{Tg:\,g\sqsubseteq f\}$.
\item~$(T)^{-}(f):=-\inf\{Tg:\,g;\,\,g\sqsubseteq f\}$.
\item~$|Tf|\leq|T|(f)$.
\end{enumerate}
\end{thm}

\section{The fragments of an abstract Uryson operator}

Let $E,F$ be vector lattices with $F$ Dedekind complete and $T\in\mathcal{U}_{+}(E,F)$.
The purpose of this section is to describe the fragments of $T$. That is
$$
\mathcal{F}_{T}=\{S\in\mathcal{U}_{+}(E,F):\,S\wedge(T-S)=0\}.
$$
Like in the linear case we consider elementary fragments.
For a subset $\mathcal{A}$ of a vector lattice $W$ we employ  the
following notation:
\begin{gather}
\mathcal{A}^{\upharpoonleft}=\{x\in W:\exists\,\,\text{a sequence}\,\,(x_{n})\subset\mathcal{A}\,\,\text{with}\,\,x_{n}\uparrow x\};\notag \\
\mathcal{A}^{\uparrow}=\{x\in W:\exists\,\,\text{a net}\,\,(x_{\alpha})\subset\mathcal{A}\,\,\text{with}\,\,x_{\alpha}\uparrow x\}.\notag
\end{gather}
The meanings of $\mathcal{A}^{\downharpoonleft}$ and $\mathcal{A}^{\downharpoonleft}$ are analogous. As usual, we also write
$$
\mathcal{A}^{\downarrow\uparrow}=(\mathcal{A}^{\downarrow})^{\uparrow};
\mathcal{A}^{\upharpoonleft\downarrow\uparrow}=((\mathcal{A}^{\upharpoonleft})^{\downarrow})^{\uparrow}.
$$
It is clear that $\mathcal{A}^{\downarrow\downarrow}=\mathcal{A}^{\downarrow}$, $\mathcal{A}^{\uparrow\uparrow}=\mathcal{A}^{\uparrow}$.
Consider a positive abstract Uryson operator $T:E\rightarrow F$, where $F$ is Dedekind complete.
Since $\mathcal{F}_{T}$ is a Boolean algebra, it is closed under finite suprema and
infima. In particular, all ``ups and downs'' of $\mathcal{F}_{T}$  are likewise closed under finite suprema and infima, and therefore they
are also directed upward and, respectively, downward.

\begin{definition}\label{def:adm}
A subset $D$ of a vector lattice $E$ is called \it{admissible} if the following conditions hold
\begin{enumerate}
\item~if $x\in D$ then $y\in D$ for every $y\in\mathcal{F}_{x}$;
\item~if $x,y\in D$, $x\bot y$ then $x+y\in D$.
\end{enumerate}
\end{definition}

Let $T\in\mathcal{U}_{+}(E,F)$ and $D\subset E$ be an admissible set. Then we define a map  $\pi^{D}T:E\rightarrow F_{+}$ by formula
$$
\pi^{D}T(x)=\sup\{Ty:\,y\sqsubseteq x,\,y\in D\},
$$
for every $x\in E$.

\begin{lemma}\label{le:01}
Let $E,F$ be vector lattices  with $F$ Dedekind complete, $\rho\in\mathfrak{B}(F)$, $T\in\mathcal{U}_{+}(E,F)$ and $D$ be an admissible set.
Then $\pi^{D}T$ is a positive abstract Uryson operator and $\rho\pi^{D}T\in\mathcal{F}_{T}$.
\end{lemma}

\begin{proof}
Let us show that $\pi^{D}(T)\in\mathcal{U}_{+}(E,F)$.
Fix $x,y\in E$ with $x\bot y$. If $z\sqsubseteq x+y$ and $z\in D$, by the Riesz decomposition property,
there exist $z_{1},z_{2}$ such that $z_{1}+z_{2}=z$, $z_{1}\bot z_{2}$ and $z_{1},z_{2}\in D$. Then we have
$$
Tz=T(z_{1}+z_{2})=T(z_{1})+T(z_{2}),
$$
and therefore $\pi^{D}(T)(x+y)\leq\pi^{D}(T)(x)+\pi^{D}(T)(y)$. Now we prove the reverse inequality.
If $z_{1}\sqsubseteq x$, $z_{1}\in D$ and  $z_{2}\sqsubseteq y$, $z_{1}\in D$ then we have $z_{1}+z_{2}\in D$. Hence
$$
T(z_{1})+T(z_{2})=T(z_{1}+z_{2})\leq\pi^{D}(T)(x+y).
$$
Taking the supremum in the left hand side we may write
$$
\pi^{D}(T)(x)+\pi^{D}(T)(y)\leq\pi^{D}(T)(x+y).
$$
Finally we have
$$
\pi^{D}(T)(x)+\pi^{D}(T)(y)=\pi^{D}(T)(x+y).
$$
Now fix an order projection $\rho$ on the
Dedekind complete vector lattice space $F$.
Then the operator $\Pi:\mathcal{U}(E,F)\rightarrow\mathcal{U}(E,F)$ define by $\Pi(T)=\rho\pi^{D}T$
satisfies $0\leq\Pi(T)\leq T$ for each $T\in\mathcal{U}_{+}(E,F)$ and $\Pi^{2}=\Pi$. Thus, by (\cite{Al},Theorem~$1.44$)
the operator $\Pi$ is an order projection on $\mathcal{U}(E,F)$. Consequently,
if $T$ is a positive abstract Uryson operator, then for each order projection $\rho$ on $F$ and each admissible set $D\subset E$
the operator $\rho\pi^{D}T$ is a fragment  of $T$.
\end{proof}
Consider some examples.

\begin{example}\label{adm-1}
Let $E$ be a vector lattice. Every order ideal in $E$ is an admissible set.
\end{example}

\begin{example}\label{adm-11}
Let $E,F$ be a vector lattices and $T\in\mathcal{U}_{+}(E,F)$. Then $\mathcal{N}_{T}:=\{e\in E:\,T(e)=0\}$  is an admissible set.
\end{example}

The following example is important for further considerations.

\begin{lemma} \label{Ex1}
Let $E$ be a vector lattice and $x\in E$. Then $\mathcal{F}_{x}$ is an admissible set.
\end{lemma}\label{le:02}

\begin{proof}
Let $y\sqsubset x$ and $z\sqsubseteq y$. Then $(x-y)\bot y$ and $(y-z)\bot z$. We may write
$$
|x-z|\wedge|z|=|x-y+y-z|\wedge|z|\leq
$$
$$
\leq(|x-y|+|y-z|)\wedge|z|\leq
$$
$$
\leq|x-y|\wedge|z|+|y-z|\wedge|z|\leq
$$
$$
\leq|x-y|\wedge|y|+|y-z|\wedge|z|=0.
$$
Now we prove (2) of Definition~\ref{def:adm}.
Let $y_{1}\sqsubseteq x$, $y_{2}\sqsubseteq x$ and  $y_{1}\bot y_{2}$. Then we have
$$
|x-y_{1}-y_{2}|\wedge|y_{1}|\leq(|x-y_{1}|+|y_{2}|)\wedge|y_{1}|\leq
$$
$$
|x-y_{1}|\wedge|y_{1}|+|y_{2}|\wedge|y_{1}|=0;
$$
$$
|x-y_{1}-y_{2}|\wedge|y_{2}|\leq(|x-y_{2}|+|y_{1}|)\wedge|y_{2}|\leq
$$
$$
|x-y_{2}|\wedge|y_{2}|+|y_{1}|\wedge|y_{2}|=0;
$$
$$
|x-y_{1}-y_{2}|\wedge|y_{1}+y_{2}|=|x-y_{1}-y_{2}|\wedge(|y_{1}|+|y_{2}|)
$$
$$
=|x-y_{1}-y_{2}|\wedge(|y_{1}|\vee|y_{2}|)=
$$
$$
=(|x-y_{1}-y_{2}|\wedge|y_{1}|)\vee(|x-y_{1}-y_{2}|\wedge|y_{2}|)=0.
$$

\end{proof}
Nonempty admissible sets $D_{1}$ and $D_{2}$ are called {\it distinct} if $D_{1}\bigcap D_{2}=\{0\}$.
If $D=\mathcal{F}_{x}$ then we denote the
operator $\pi^{D}T$ by $\pi^{x}T$.
Recall that a family of mutually disjoint order projections  $(\rho_{\xi})_{\xi\in\Xi}$ is said to be {\it full}
if $\bigvee\limits_{\xi\in\Xi}(\rho_{\xi})_{\xi\in\Xi}=Id_{F}$.
Any fragment of the form $\sum_{i=1}^{n}\limits\rho_{i}\pi^{x_{i}}T$, $n\in\Bbb{N}$, where $\rho_{1},\dots,\rho_{n}$ is a  finite family of mutually disjoint order projections in $F$, like in  the linear case is
called an {\it elementary} fragment of $T$. The set of all elementary fragments of $T$ we denote by $\mathcal{A}_{T}$.

\begin{lemma}\label{le:1}
Let $E,F$ be vector lattices, $F$ Dedekind complete with a
filter of weak order units $\mathfrak{A}_{F}$, $S,T\in\mathcal{U}_{+}(E,F)$.
Then $S\bot T$ if and only if  for every $x\in E$, $\varepsilon>0$, $\bold{1}\in\mathfrak{A}_{F}$
there exists a full family  $(\rho_{\xi})_{\xi\in\Xi}$ of mutually disjoint order projections on $F$,
and a family $(x_{\xi})_{\xi\in\Xi}$ of fragments of $x$ such that
$\rho_{\xi}\pi^{x_{\xi}}T(x)\leq\varepsilon\bold{1}$ and
$\rho_{\xi}(S-\rho_{\xi}\pi^{x_{\xi}}S)x\leq\varepsilon\bold{1}$ for every $\xi\in\Xi$.
\end{lemma}

\begin{proof}
By Theorem~\ref{th-1}
$$
0=(T\wedge S)x=\inf\{Tx_{1}+Sx_{2}:\,x=x_{1}+x_{2},\,x_{1}\bot x_{2}\}.
$$
Observe that for every $y\in E$ we have $\pi^{y}Ty=Ty$. Take an element $u\in F$, $u=Sx\wedge Tx+\rho^{\bot}(Sx+Tx)$, where $\rho$ is the order projection on the band $\{Sx\wedge Tx\}^{\bot\bot}$. It is not difficult to check that $Sx+Tx\in\{u\}^{\bot\bot}$, $\rho_{1}u\leq Tx$ and $\rho_{2}u\leq Sx$, where $\rho_{1}$, $\rho_{2}$ are order projections on the bands $\{Tx\}^{\bot\bot}$ and $\{Sx\}^{\bot\bot}$ respectively. Then we have
$$
\rho_{\xi}(Tx_{\xi}+S(x-x_{\xi}))\leq\varepsilon u
$$
for a certain  full family $(\rho_{\xi})_{\xi\in\Xi}$ of mutually disjoint order projections on $F$,
and the family $(x_{\xi}+(x-x_{\xi})=x)_{\xi\in\Xi}$ of the partitions of the element $x$. Consequently,
$\rho_{\xi}Tx_{\xi}\leq\rho_{1}\varepsilon u\leq\varepsilon Tx$ and $\rho_{\xi}S(x-x_{\xi})\leq\varepsilon Sx$.
The  converse assertion is obvious.
\end{proof}

\begin{lemma}\label{le:2}
Let $E,F$ be vector lattices, $F$ Dedekind complete with a
filter of the weak order units $\mathfrak{A}_{F}$, $T\in\mathcal{U}_{+}(E,F)$.
If $S\in\mathcal{F}_{T}$ then for every $x\in E$, $\varepsilon>0$, $\bold{1}\in\mathfrak{A}_{F}$
there exists a full family  $(\rho_{\xi})_{\xi\in\Xi}$ of mutually disjoint order projections on $F$,
and a family $(x_{\xi})_{\xi\in\Xi}$ of fragments of $x$, such that
$\rho_{\xi}|(S-\rho_{\xi}\pi^{x_{\xi}}T)|x\leq\varepsilon\bold{1}$ for every $\xi\in\Xi$.
\end{lemma}

\begin{proof}
Using Lemma~\ref{le:1} and taking operators $S,T-S$ and the weak order unit $\frac{1}{2}\bold{1}$  we have
$$
\rho_{\xi}|(S-\rho_{\xi}\pi^{x_{\xi}}T)|x\leq\rho_{\xi}|(S-\rho_{\xi}\pi^{x_{\xi}}S)|x+
\rho_{\xi}|(\rho_{\xi}\pi^{x_{\xi}}S-\rho_{\xi}\pi^{x_{\xi}}T)|x=
$$
$$
=\rho_{\xi}|(S-\rho_{\xi}\pi^{x_{\xi}}S)|x+
\rho_{\xi}|(\rho_{\xi}\pi^{x_{\xi}}(T-S))|x.
$$
\end{proof}

\begin{lemma}\label{le:3}
Let $E,F$ be the same as in Lemma~\ref{le:1},  $T\in\mathcal{U}_{+}(E,F)$ and $S\in\mathcal{F}_{T}$. Then
\begin{enumerate}
\item for every $x\in E$, $\varepsilon>0$, $\bold{1}\in\mathfrak{A}_{F}$ there exists
$G_{x}\in\mathcal{A}_{T}^{\uparrow}$, so that $|S-G_{x}|x\leq\varepsilon\bold{1}$;
\item for every $x\in E$  there exists
$R_{x}\in\mathcal{A}_{T}^{\uparrow\downharpoonleft}$, so that $|S-R_{x}|x=0$.
\end{enumerate}
\end{lemma}

\begin{proof}
Let us to prove $(1)$. By Lemma~\ref{le:2} there exists a full family  $(\rho_{\xi})_{\xi\in\Xi}$ of mutually disjoint order projections on $F$,
and a family $(x_{\xi})_{\xi\in\Xi}$
of fragments of  $x$ such that
$\rho_{\xi}|(S-\rho_{\xi}\pi^{x_{\xi}}T)|x\leq\varepsilon\bold{1}$ for every $\xi\in\Xi$.
By $\Theta$ we denote the system of all finite subsets of $\Xi$. It is an ordered by inclusion set.
Surely, $\Theta$ is a directed set. For every $\theta\in\Theta$ set
$G_{\theta}=\sum\limits_{\theta\in\Theta}\rho_{\xi}\pi^{x_{\xi}}T$. The net $(G_{\theta})_{\theta\in\Theta}$ is increasing.
Let $G_{x}=\sup(G_{\theta})_{\theta\in\Theta}$. Then $G_{x}\in\mathcal{A}_{T}^{\uparrow}$ and we may write
$$
\rho_{\xi}|(S-G_{\theta})|x=
\rho_{\xi}|(S-\sum\limits_{\theta\Theta}\rho_{\xi}\pi^{x_{\xi}}T)|x\leq\varepsilon\bold{1}
$$
for every $\xi\in\Xi$ and every $\theta\geq\{\xi\}$. Therefore $\rho_{\xi}|S-G_{x}|x\leq\varepsilon\bold{1}$ for every
$\xi\in\Xi$ and $|S-G_{x}|x\leq\varepsilon\bold{1}$.

Now we prove $(2)$. Fix any $\bold{1}\in\mathfrak{A}_{F}$. For $\varepsilon_{n}=\frac{1}{2^{n}}$ there exists
$G_{x}^{n}\in\mathcal{A}_{T}$ such that $|S-G_{x}^{n}|x\leq\frac{1}{2^{n}}\bold{1}$. Let $C_{x}^{k}=\bigvee\limits_{n=k}^{\infty}G_{x}^{n}$ and
$C_{x}^{k,i}=\bigvee\limits_{n=k}^{n=k+i}G_{x}^{n}$. Since $\mathcal{A}_{T}$ is a subalgebra of $\mathcal{F}_{T}$, one has $C_{x}^{k,i}\in\mathcal{A}_{T}^{\uparrow}$  and
$C_{x}^{k,i}\uparrow C_{x}^{k}\in\mathcal{A}_{T}^{\uparrow\upharpoonleft}=\mathcal{A}_{T}^{\uparrow}$. Then we have
$$
|S-C_{x}^{k,i}|x=|S-\bigvee\limits_{n=k}^{n=k+i}G_{x}^{n}|x=|\bigvee\limits_{n=k}^{n=k+i}(S-G_{x}^{n})|x\leq
$$
$$
\leq\sum\limits_{n=k}^{n=k+i}|S-G_{x}^{n}|x\leq\sum\limits_{n=k}^{\infty}\frac{1}{2^{n}}\bold{1}\leq\frac{1}{2^{k-1}}\bold{1}.
$$
So we may write  $|S-C_{x}^{k}|\leq\frac{1}{2^{k-1}}\bold{1}$. The sequence $(C_{x}^{k})$ is decreasing. Let $R_{x}=\inf{C_{x}^{k}}$. Then $R_{x}\in\mathcal{A}_{T}^{\uparrow\downharpoonleft}$ and $|S-R_{x}|x=0$.
\end{proof}
Remark that $R_{x}y=0$ for every $y$ such that $\mathcal{F}_{x}$ and $\mathcal{F}_{y}$ are distinct admissible sets.
Moreover, if $y\in\mathcal{F}_{x}$ and $|S-R_{x}|x=0$ we can write $0\leq|S-R_{x}|y\leq|S-R_{x}|x=0$, and therefore $|S-R_{x}|y=0$
for every $y\in\mathcal{F}_{x}$.

\begin{lemma}\label{le:6}
Let $E$ be a vector lattice and $D\subset E$. Then there exists a family $(y_{\lambda})_{\lambda\in\Lambda}$ of
elements of $D$ so that   $(\mathcal{F}_{y_{\lambda}})_{\lambda\in\Lambda}$ are pairwise distinct admissible set and
$$
D=\bigcup\limits_{\alpha\in\Lambda}\mathcal{F}_{y_{\alpha}}.
$$
\end{lemma}

\begin{proof}
For the proof we apply the method of transfinite induction. Let $\beta$ be the least ordinal of cardinality $|D|$ and let $(x_\alpha)_{\alpha < \beta}$ be any well ordering of $D$.  Let $y_0=x_0$ and $\lambda<\beta$ be a cardinal number. Assume   that $\mathcal{F}_{y_{\mu}}$ and $\mathcal{F}_{y_{\mu'}}$ are distinct admissible sets for every $\mu\neq\mu';\,\mu,\mu'<\lambda$.
Take the element $x_{\lambda}$ and consider the set
$$
\overline{\mathcal{F}}_{x_{\lambda}}:=
\{z\in\mathcal{F}_{x_{\lambda}},\,z\notin\bigcup\limits_{\mu<\lambda}\mathcal{F}_{y_{\mu}}\}.
$$
This set is a subset of $\mathcal{F}_{x_{\lambda}}$ and consequently is partially ordered by inclusion and every linearly ordered chain is bounded. Therefore
there exists a maximal element which denoted by $y_{\lambda}$. Taking into account the definition of the element $y_{\lambda}$ we have that $\mathcal{F}_{y_{\lambda}}$ and $\mathcal{F}_{y_{\mu}}$ are distinct admissible sets for every $\mu<\lambda$. Then we may write  $D=\bigcup\limits_{\alpha\leq\beta}\mathcal{F}_{y_{\alpha}}$ and
$(y_{\alpha})$ is a desired family.
\end{proof}

\begin{thm}
Let $E,F$ be vector lattices, $F$ Dedekind complete with a
filter of weak order units $\mathfrak{A}_{F}$, $T\in\mathcal{U}_{+}(E,F)$ and $S\in\mathcal{F}_{T}$.
Then $S\in\mathcal{A}_{T}^{\uparrow\downharpoonleft\uparrow}$.
\end{thm}

\begin{proof}
Let $D=\{x\in E:\,S(x)\neq 0\}$ and consider the family $(y_{\lambda})_{\lambda\in\Lambda}$ like in Lemma~\ref{le:6}.
By $\Theta$ we denote the system of all finite subsets of $\Lambda$. It is an ordered by inclusion set.
It is clear that, $\Theta$ is a directed set. For every $\theta\in\Theta$ denote by
$R_{\theta}=\bigvee\limits_{\lambda\in\theta}R_{y_{\lambda}}$. The net $(R_{\theta})_{\theta\in\Theta}$ is increasing.
Let $R=\sup(R_{\theta})_{\theta\in\Theta}$. Then $R\in\mathcal{A}_{T}^{\uparrow\downharpoonleft\uparrow}
$ and $|S-R|x=0$ for every $x\in E$. Therefore $R=S$ and $S\in\mathcal{A}_{T}^{\uparrow\downharpoonleft\uparrow}$.
\end{proof}
Remark that for linear positive operators the same theorem and its  modifications were proven by  de Pagter, Aliprantis and Burkinshaw, Kusraev and Strizhevski   in \cite{Al-1,KS,Pag}.

\section{Dominated  Uryson operator}

In  this section   we  consider  a  wide class of orthogonally additive operators acting from a lattice-normed space $(V,E)$ to another lattice-normed space $(W,F)$, called  dominated Uryson operators,  and investigate some properties of these operators. In particular, we find a formula for calculation the exact dominant of a dominated operator and show that the vector norm of a dominated operator is decomposable. At first, dominated Uryson operators were introduced and studied in \cite{Ku-1}.
But our approach is different. We consider a more wider class of dominants. As a result, the   set of dominants is Dedekind complete vector lattice, the space of the dominated Uryson operators is decomposable, etc.

\begin{definition} \label{def:even}
Let $E$ be a vector lattice  and $X$ a vector space. An orthogonally additive map $T:E\to X$ is called  even if $T(x)=T(-x)$ for every $x\in E$. If $E,F$ are vector lattices, the set of all even abstract Uryson operators from $E$ to $F$ we denote by $\mathcal{U}^{ev}(E,F)$.
\end{definition}
If $E,F$ are vector lattices with $F$ Dedekind complete, the space $\mathcal{U}^{ev}(E,F)$ is not empty. Indeed, for every $T\in\mathcal{U}(E,F)$ by (\cite{Maz-1},Proposition~3.4) there exists an even operator $\widetilde{T}\in U_{+}^{ev}(E,F)$ which is defined by the formula,
$$
\widetilde{T}f=\sup\{|T|g:\,|g|\leq|f|\}.
$$

\begin{lemma}\label{lemma-ev}
Let $E,F$ be vector lattices with $F$ Dedekind complete. Then $\mathcal{U}^{ev}(E,F)$ is a Dedekind complete sublattice of $\mathcal{U}(E,F)$.
\end{lemma}
\begin{proof}
It is clear that $\mathcal{U}^{ev}(E,F)$ is a vector subspace. Let us prove that $\mathcal{U}^{ev}(E,F)$ is a sublattice of $\mathcal{U}^{ev}(E,F)$. Fix $T\in\mathcal{U}^{ev}(E,F)$ and take $T^{+}$. Observe that $\{y:\,y\in\mathcal{F}_{(-x)}\}=\{-y:\,y\in\mathcal{F}_{x}\}$ for an arbitrary $x\in E$. Then by \ref{th-1} we have
$$
T^{+}(-x)=\sup\{Ty:\,y\sqsubseteq(-x)\}=\sup\{T(-y):\,y\sqsubseteq x\}=
$$
$$
=\sup\{T(y):\,y\sqsubseteq x\}=T^{+}(x),\,\,(x\in E).
$$
Hence $\mathcal{U}^{ev}(E,F)$ is a vector sublattice. To prove that the
vector sublattice $\mathcal{U}^{ev}(E,F)$ is Dedekind complete, assume that
$(T_{\alpha})\subset\mathcal{U}^{ev}(E,F)$ and $0\leq T_{\alpha}\uparrow T$, where $T\in\mathcal{U}_{+}(E,F)$. We have to prove that $T=\sup_{\alpha}T_{\alpha}$ is also an even positive abstract Uryson operator. For an arbitrary $x\in E$ we have
$$
T(-x)=\sup_{\alpha}T_{\alpha}(-x)=\sup_{\alpha}T_{\alpha}(x)=T(x).
$$
Hence $\mathcal{U}^{ev}(E,F)$ is Dedekind complete.
\end{proof}

\begin{definition} \label{def:dom}
Let $(V,E)$ and $(W,F)$ be lattice-normed spaces. A map $T:V\to W$ is called {\it orthogonally additive} if $T(u+v)=Tu+Tv$ for every $u,v\in V,\,u\bot v$. An orthogonally additive map $T:V\to W$ is called a {\it dominated Uryson} operator if there exists $S\in\mathcal{U}_{+}^{ev}(E,F)$ such that $\ls Tv\rs\leq S\ls v\rs$ for every $v\in V$. In this case we say that $S$ is a \textit{dominant} for $T$. The set of all dominants of the operator $T$ is denoted by $\text{Domin}(T)$. If there is the least element in $\text{Domin}(T)$ with respect to the order induced by $\mathcal{U}_{+}^{ev}(E,F)$ then it is called the {\it least} or the {\it exact dominant} of $T$ and is denoted by
$\ls T\rs$. The set of all dominated Uryson operators from $V$ to $W$ is denoted by $\mathcal{D}_{U}(V,W)$.
\end{definition}

\begin{example} \label{Ex0}
Let $X,Y$ be normed spaces.  Consider the lattice-normed spaces $(X,\Bbb{R})$ and $(Y,\Bbb{R})$. Then a given map $T:X\to Y$ is an element of $\mathcal{D}_{U}(X,Y)$ if and only if there exists  an even function $f:\Bbb{R}\to\Bbb{R}_{+}$ such that $\|Tx\|\leq f(\|x\|)$ for every $x\in X$.
\end{example}

\begin{example} \label{Ex1}
Let $E,F$ be vector lattices with $F$ Dedekind complete.  Consider the lattice-normed spaces $(E,E)$ and $(F,F)$ where the lattice valued  norms coincide with the modules. We may show that the vector space $\mathcal{D}_{U}(E,F)$ coincide with $\mathcal{U}(E,F)$. Indeed, if $T\in\mathcal{D}_{U}(E,F)$, then there exists $S\in\mathcal{U}_{+}^{ev}(E,F)$ such that
$|Tx|\leq S|x|$ for every $x\in E$. Thus, $T$ is order bounded. If $T\in\mathcal{U}(E,F)$ then by (\cite{Maz-1}, Proposition~3.4) there exists $S\in\mathcal{U}_{+}^{ev}(E,F)$, so that $|Tf|\leq S(f)\leq S(|f|)$ and therefore $T\in\mathcal{D}_{U}(E,F)$.
\end{example}

\begin{example} \label{Ex2}
Let $(A,\Sigma,\mu)$  be a finite complete measure space, $E$ an order ideal in $L_{0}(\mu)$ and  $X$ a Banach space. Let $N:A\times X\to X$ be a function satisfying the following conditions:
\begin{enumerate}
  \item[$(C_{0})$] $N(t,0)=0$ for $\mu$-almost all $t\in A$;
  \item[$(C_{1})$] $N(\cdot,x)$ is Bochner $\mu$-measurable for all $x\in X$;
  \item[$(C_{2})$] $N(t,\cdot)$ is continuous with respect to the norm of $X$  for $\mu$-almost all $t\in A$.
  \item[$(C_{3})$] There exists a measurable function $M:A\times\Bbb{R}\to\Bbb{R}_{+}$,
  so that $M(t,r)=M(t,-r)$ for $\mu$-almost all $t\in A$, $r\in\Bbb{R}$ and
   $$
   \sup\limits_{\|x\|\leq r}\|N(t,x)\|\leq M(t,r)\,\, \text{for all}\,\, t\times r\in A\times\Bbb{R}.
   $$
\end{enumerate}
By $\text{Dom}(N)$ we denote the set of the Bochner $\mu$-measurable vector-function $f:A\to X$, so that $N(\cdot,f(\cdot))\in L_{1}(\mu,X)$. If $E(X)\subset\text{Dom}(N)$ and $M(\cdot,g(\cdot))\in L_{1}(\mu)$ for every $g\in E$, we may define an orthogonally additive  operator $T:E(X)\to X$ by the formula
$$
Tf:=\int_{A}N(t,f(t))\,d\mu(t).
$$
Let us show that $T\in\mathcal{D}_{U}(E(X),X)$. Indeed
$$
\ls Tf\rs=\|Tf\|=\|\int_{A}N(t,f(t))\,d\mu(t)\|\leq\int_{A}\|N(t,f(t))\|\,d\mu(t)\leq
$$
$$
\leq\int_{A}M(t,\|f(t)\|)\,d\mu(t)=S\ls f\rs,
$$
where $S:E\to\Bbb{R}_{+}$ is the integral Uryson operator, $Se=\int_{A}M(t,e(t))\,d\mu(t)$ and $S$ is a dominant for $T$.
\end{example}

\begin{thm}\label{dom-01}
Let $(V,E)$, $(W,F)$ be lattice-normed spaces with $V$ decomposable and $F$ Dedekind complete. Then every dominated Uryson operator $T:V\to W$ has an exact dominant $\ls T\rs$.
\end{thm}
\begin{proof}
At first we observe that $\text{Domin}(T)$ is a lower semilattice in $\mathcal{U}_{+}^{ev}(E,F)$. It means that if $S_{1},S_{2}\in\text{Domin}(T)$ then $S_{1}\wedge S_{2}\in\text{Domin}(T)$. Indeed, if $\ls u\rs=e_{1}+e_{2}$, where
$u\in V$ and $e_{1}\bot e_{2}$, then we have $u=u_{1}+u_{2}$, $\ls u_{i}\rs=e_{i}$, $i\in\{1,2\}$. Consequently the following inequality holds
$$
\ls Tu\rs=\ls Tu_{1}+Tu_{2}\rs\leq\ls Tu_{1}\rs+\ls Tu_{2}\rs\leq S_{1}e_{1}+S_{2}e_{2}.
$$
Take the infimum over $e_{1}$ and $e_{2}$, $e_{1}+e_{2}=\ls u\rs$, $e_{1}\bot e_{2}$ yields $\ls Tu\rs\leq(S_{1}\wedge S_{2})\ls u\rs$. Thus $S_{1}\wedge S_{2}\in\text{Domin}(T)$ and the set $\text{Domin}(T)$ is downward directed. Then the infimum $R=\inf\{S:\,S\in\text{Domin}(T)\}$ can be calculated pointwise on the cone $E_{+}$. It follows that
$$
\ls Tu\rs\leq\inf\{S\ls u\rs:\,S\in\text{Domin}(T)\}=R\ls u\rs,\,(u\in V).
$$
Hence $R\in\text{Domin}(T)$ and $R=\ls T\rs$.
\end{proof}
For further consideration we introduce the following set
$$
\widetilde{E}_{+}=\{e\in E_{+}:\,e=\bigsqcup\limits_{i=1}^{n}\ls v_{i}\rs;\,v_{i}\in V;\,n\in\Bbb{N}\}.
$$

\begin{lemma}\label{dom-1}
Let $(V,E)$ be lattice-normed spaces with $V$ $d$-decomposable. Then  $\widetilde{E}_{+}$ is an admissible set.
\end{lemma}

\begin{proof}
Let us prove the first assertion from \ref{def:adm}. Fix $e\in\widetilde{E}_{+}$ and $e_{0}\sqsubseteq e$.
Then there exist $w_{1},\dots,w_{n}$ and a family of mutually disjoint elements of $V$ such that $e=\bigsqcup\limits_{i=1}^{n}\ls w_{i}\rs$ and $e=e_{0}+(e-e_{0})$. Then by the $d$-decomposability of $V$, there exist two families $(v_{1},\dots,v_{n})$ and $(u_{1},\dots,u_{n})$ of mutually disjoint elements of $V$ such that
$$
w_{i}=v_{i}+u_{i};\,v_{i}\bot u_{i};\,i\in\{1,\dots,n\};
$$
$$
e_{0}=\bigsqcup\limits_{i=1}^{n}\ls v_{i}\rs;\,(e-e_{0})=\bigsqcup\limits_{i=1}^{n}\ls u_{i}\rs.
$$
Item (2) of Definition~\ref{def:adm}  is obvious.
\end{proof}

\begin{thm}\label{dom-2}
Let $(V,E)$, $(W,F)$ be  the same as  in Theorem~\ref{dom-01}. Then the exact dominant of a dominated Uryson operator $T:V\to W$ can be calculated by the following formulas
\begin{enumerate}
  \item[$(1)$] $\ls T\rs(e)=\sup\Big\{\sum\limits_{i=1}^{n}\ls Tu_{i}\rs:\,\coprod_{i=1}^{n}\ls u_{i}\rs=e,\, n\in\Bbb{N}\Big\}$\,$(e\in\widetilde{E}_{+})$;
  \item[$(2)$] $\ls T\rs(e)=\sup\Big\{\ls T\rs(e_{0}):\,e_{0}\in\widetilde{E}_{+},\, e_{0}\sqsubseteq e\Big\}$;\,$(e\in E_{+})$
  \item[$(3)$] $\ls T\rs(e)=\ls T\rs(e_{+})+\ls T\rs(e_{-})$,\,$e\in E$.
\end{enumerate}
\end{thm}

\begin{proof}
Take $e\in\widetilde{E}_{+}$ and denote by $Re$ the right  side of the  formula $(1)$.  For every mutually disjoint family $u_{1},\dots,u_{n}$ the elements of $V$  we may write
$$
\sum\limits_{i=1}^{n}\ls Tu_{i}\rs\leq\ls T\rs\Big(\sum\limits_{i=1}^{n}\ls u_{i}\rs \Big)=\ls T\rs(e).
$$
Since the vector lattice $F$ is  Dedekind  complete  the map $R:\widetilde{E}_{+}\to F$ is well defined. If $u_{1},\dots,u_{n}$, $v_{1},\dots,v_{m}$ are mutually disjoint families of elements of $V$, so that $\coprod\limits_{i=1}^{n}\ls u_{i}\rs=e$ and $\coprod\limits_{i=1}^{n}\ls v_{i}\rs=f$, $e\bot f$ then
$e,f\in\widetilde{E}_{+}$ and we may write
$$
\sum\limits_{i=1}^{n}\ls Tu_{i}\rs+\sum\limits_{k=1}^{m}\ls Tv_{k}\rs\leq R(e+f).
$$
Passing to the supremum over all mutually disjoint families $u_{1},\dots,u_{n}$ and $v_{1},\dots,v_{m}$ we obtain $R(e)+R(f)\leq R(e+f)$. Let $e+f=\sum\limits_{k=1}^{n}\ls w_{k}\rs$. By the decomposability of the
lattice valued norm in $V$, there exist finite families of mutually disjoint elements $e_{1},\dots, e_{n}\in \widetilde{E}_{+}$, $f_{1},\dots,f_{m}\in \widetilde{E}_{+}$, and $u_{1},\dots, u_{n}\in V$, $v_{1},\dots,v_{n}\in V$, such that
$$
f=f_{1}+\dots+f_{n};\,e=e_{1}+\dots+e_{n};
$$
$$
w_{k}=u_{k}+v_{k};\,\ls u_{k}\rs=e_{k};\,\ls v_{k}\rs=f_{k};\,k\in\{1,\dots,n\}.
$$
Then we have
$$
\sum\limits_{k=1}^{n}\ls Tw_{k}\rs\leq\sum\limits_{k=1}^{n}\ls Tv_{k}\rs+
\sum\limits_{k=1}^{n}\ls Tu_{k}\rs\leq R(e+f).
$$
Taking the supremum over all families of mutually disjoint elements of $V$, we obtain the reverse inequality
$R(e+f)\leq R(e)+R(f)$. Thus $R$ is an orthogonally additive operator. Now extend $R$ from $\widetilde{E}_{+}$ to $E_{+}$ by letting
$$
Re=\sup\{Re_{0}:\,e_{0}\sqsubseteq e;\,e_{0}\in\widetilde{E}_{+}\}.
$$
Since $\widetilde{E}_{+}$ is  an admissible set, the extended operator is well defined. Since $\ls T\rs$ is an order bounded operator, $\widetilde{E}_{+}$ is an admissible set and $F$ is Dedekind complete, the extended operator is orthogonally additive and there exists $M\in F_{+}$,
so that $0\leq f\leq M$, for every $f\in(\ls T\rs[0,e])$. So we have $Re_{0}\leq\ls T\rs(e_{0})\leq M$
for every $e_{0}\sqsubseteq e$. Therefore, the supremum in the definition of $Re$ exists. Moreover, $Re\leq\ls T\rs(e)$.
Finally, letting $Re=R(e_{+})+R(e_{-})$ for $e\in E$ we obtain some even positive abstract Uryson operator $R:E\rightarrow F$ and $R\leq\ls T\rs$. On the other hand, for  $v\in V$, we have $\ls Tv\rs\leq R\ls v\rs$. Thus, $R$ is a dominant for $T$ and, hence $\ls T\rs\leq R$. Finally, we obtain $R=\ls T\rs$.
\end{proof}
The following corollary of Theorem~\ref{dom-2} is often useful.
\begin{cor}
Let $(V,E)$, $(W,F)$ be the same as in Theorem~\ref{dom-2}. Then an
orthogonally additive operator $T:V\to W$ is dominated if and only if the set
$$
\mathcal{O}(e)=\Big\{\sum\limits_{i=1}^{n}\ls Tv_{i}\rs:\,\bigsqcup\limits_{i=1}^{n}\ls v_{i}\rs\sqsubseteq e,\,n\in\Bbb{N}\Big\}
$$
is order bounded for every $e\in E_{+}$.
\end{cor}
\begin{proof}
Necessity is obvious. Suppose that $\mathcal{O}(e)$ is an order bounded set for every $e\in E_{+}$. Denote the $\sup\{\mathcal{O}(e)\}$ by $O(e)$. Define an orthogonally additive operator $R:E_{+}\to F$ in the same way as in the Theorem~\ref{dom-2}, By order boundedness the $\mathcal{O}(e)$ this definition is correct, moreover $Re\leq O(e)$ for every $e\in E_{+}$. Suppose that for some $v_{1},\dots, v_{n}\in V$, $u_{1},\dots, u_{m}\in V$, $e\in E_{+}$, we have
$$
(\ls v_{1}\rs+\dots +\ls v_{n}\rs)\sqsubseteq e=\ls u_{1}\rs+\dots +\ls u_{m}\rs.
$$
Then the element $f:=e-(\ls v_{1}\rs+\dots +\ls v_{n}\rs)$ can be represented as $f=\bigsqcup\limits_{i=1}^{m}f_{i}$, where $0\leq f_{i}\leq\ls u_{i}\rs$, $i\in\{1,\dots,m\}$. By the decomposability of $V$, there exist $w_{1},\dots, w_{m}\in V$ such that $\ls w_{i}\rs=f_{i}$, $i\in\{1,\dots, m\}$. So, for $p=\ls Tv_{1}\rs+\dots+\ls Tv_{n}\rs$ we have
$$
p\leq p+\ls Tw_{1}\rs+\dots+\ls Tv_{m}\rs\leq Re.
$$
The element $p\in\mathcal{O}$ is arbitrary, and therefore $Oe\leq Re$ for every $e\in\sum\widetilde{E}_{+}$. We also observe that $Oe=\sup\{Re_{0}:\,e_{0}\sqsubseteq e;\,e_{0}\in\widetilde{E}_{+}\}$. Thus, the operators $R$ and $O$ coincide on $\widetilde{E}_{+}$. Finally, using the inequality $\ls Tv\rs\leq O\ls v\rs$ for every $v\in V$ we complete the proof.
\end{proof}

\section{Decomposability of the space of dominated  Uryson operator}

In this section we establish that the set of all dominated Uryson operators is
a Banach-Kantorovich space with respect to the the dominant norm.

\begin{thm} \label{BK}
Let $(V,E)$, $(W,F)$ be lattice-normed spaces with $V$ decomposable and $W$ a Banach-Kantorovich space. Then the space of all dominated Uryson operators $\mathcal{D}_{U}(V,W)$ is a Banach-Kantorovich space.
\end{thm}

First we need the following lemma.

\begin{lemma}\label{domspace-1}
Let $(V,E)$, $(W,F)$ be lattice-normed spaces with $V$ decomposable and $F$ Dedekind complete. Then the set of all
dominated Uryson operator $\mathcal{D}(V,W)$ with the map
$\ls\cdot\rs:\mathcal{D}(V,W)\to\mathcal{U}_{+}^{ev}(E,F)$ is a lattice-normed space.
\end{lemma}

\begin{proof}
As we saw in Lemma~\ref{dom-1}, in the case of a decomposable lattice-normed space $V$ and a Dedekind complete vector lattice $F$, the dominant norm $\ls T\rs\in\mathcal{U}_{+}^{ev}(E,F)$ is well defined for each dominated Uryson operator $T\in\mathcal{D}(V,W)$. So, the set $\mathcal{D}(V,W)$ with the map $\ls\cdot\rs:\mathcal{D}(V,W)\to\mathcal{U}_{+}^{ev}(E,F)$ is a lattice-normed space too. The first and second axioms of \ref{lat-norm} are  obvious. Now, for every $T_{1}, T_{2}\in\mathcal{D}(V,W)$ and $v\in V$ we may write
$$
\ls(T_{1}+T_{2})v\rs\leq\ls T_{1}v\rs+\ls T_{2}v\rs\leq S_{1}\ls v\rs+ S_{2}\ls v\rs=
$$
$$
=(S_{1}+S_{2})\ls v\rs,
$$
where $S_{1}, S_{2}$ are dominants for $T_{1}$ and $T_{2}$ respectively.
Therefore $\text{Domin}(T_{1}+T_{2})\supset\text{Domin}(T_{1})+\text{Domin}(T_{2})$ and the third axiom of Definition~\ref{lat-norm} is also valid.
\end{proof}

\begin{lemma}\label{domspace-2}
Let $(V,E)$, $(W,F)$ be lattice normed spaces with $V$ decomposable and $W$ $(bo)$-complete. Then the lattice-normed space $\mathcal{D}(V,W)$ is $(bo)$-complete.
\end{lemma}

\begin{proof}
Let $(T_{\alpha})$ be $bo$-fundamental net in $\mathcal{D}(V,W)$. It means that for
$\alpha,\beta\geq\gamma$ we have $\ls T_{\alpha}-T_{\beta}\rs\leq S_{\gamma}$, where the net is
decreasing and converges to zero in $\mathcal{U}^{ev}(E,F)$. Then we may write
$$
\ls T_{\alpha}v-T_{\beta}v\rs\leq S_{\gamma}(\ls v\rs).\,\,\,(\star)
$$
Hence the net $(T_{\alpha}v)$ is also $bo$-fundamental in $(W,F)$ for every $v\in V$.
Since $W$ is $bo$-complete, there exists an orthogonally additive
operator $T:V\to W$ defined by the formula $Tv=\bolim_{\alpha}Tv_{\alpha}$.
Passage to the limit over $\alpha$ in the $(\star)$  gives $\ls Tv-T_{\beta}v\rs\leq S_{\gamma}(\ls v\rs)$. Thus, we have
$$
\ls Tv\rs\leq\ls Tv-T_{\beta}v\rs+\ls T_{\beta}v\rs\leq S_{\gamma}(\ls v\rs)+\ls T_{\beta}\rs\ls v\rs,
$$
and the operator $T$ is dominated. Let us show that $T=\bolim_{\alpha}T{\alpha}$. Fix $e\in E_{+}$ and take $v_{1},\dots, v_{n}\in V$ so that $(\ls v_{1}\rs+\dots+\ls v_{n}\rs)\sqsubseteq e$. Then
we may write
$$
\sum\limits_{i=1}^{n}\ls T_{\alpha}v_{i}-T_{\beta}v_{i}\rs\leq\sum\limits_{i=1}^{n} S_{\gamma}(\ls v_{i}\rs)\leq S_{\gamma}(e).
$$
Passing to the order limit over $\alpha$ and taking the supremum  over all finite families $(v_{1},\dots, v_{v})$ we have $\ls T-T_{\beta}\rs\leq S_{\gamma}$ for all $\alpha\geq\gamma$ and therefore finishing the proof.
\end{proof}

\begin{lemma}(\cite{Ku},Proposition 2.1.2.3).\label{decompos}
Let $(V,E)$ be a lattice-normed space with $V$ decomposable. Then for a pair of disjoint elements $e_{1}, e_{2}\in E_{+}$
the decomposition $v=v_{1}+v_{2}$, where $\ls v_{1}\rs=e_{1}$, $\ls v_{2}\rs=e_{2}$ is unique.
\end{lemma}

\begin{lemma}\label{projection}
Let $E,F$ be vector lattices with $F$ Dedekind complete, and let $D\subset E$ be an admissible set.
Then for every $S\in\mathcal{U}_{+}(E,F)$, $e\in D$ the following equality holds
$$
(\pi^{D})^{\bot}Se=0.
$$
\end{lemma}

\begin{proof}
By definition, $(\pi^{D})Se=\sup\{Se_{0}:\,e_{0}\sqsubseteq e,\,e_{0}\in D\}$. Then for an arbitrary $e\in D$
we have $Se=(\pi^{D})Se$. Therefore,
$$
(\pi^{D})^{\bot}Se=(\pi^{D})^{\bot}\circ(\pi^{D})Se=0.
$$
\end{proof}

\begin{lemma}\label{domspace-th}
Let $(V,E)$, $(W,F)$ be lattice normed spaces with $V$ decomposable and $W$ $(bo)$-complete. Suppose $D$ is an arbitrary admissible set in $E$. For every $T\in\mathcal{D}_{U}(V,W)$ there exists a unique operator $\pi^{D}T\in\mathcal{D}_{U}(V,W)$ such that $\ls\pi^{D}T\rs=\pi^{D}\ls T\rs$ and
$\ls T-\pi^{D}T\rs=\ls T\rs-\pi^{D}\ls T\rs$.
\end{lemma}

\begin{proof}
Denote $\Psi:=\ls T\rs$. Take an element $v\in V$. We are going to construct a net $(v_{\alpha})_{\alpha\in\Lambda}$ (which depends on the $D$) for $v$
with the following properties: $\olim_{\alpha}\pi^{D}\Psi\ls v-v_{\alpha}\rs=0$ and
$\olim_{\alpha}(\pi^{D})^{\bot}\Psi\ls v_{\alpha}\rs=0$. We must note, if $e\in D$, then $(\pi^{D})^{\bot}Se=0$ for every $S\in\mathcal{U}_{+}(E,F)$. This net  can be constructed by the following procedure. Assign $e_{\alpha}:=\alpha\sqsubseteq\ls v\rs,\,\alpha\in D$ and $f_{\alpha}=\ls v\rs-e_{\alpha}$. Observe that $e_{\alpha}\bot f_{\alpha}$ for every $\alpha$ and $e_{\alpha}\sqsubseteq e_{\beta}$, $\alpha,\beta\in\Lambda$, $\beta\geq\alpha$. By the decomposability of $V$, there exists a net $(v_{\alpha})_{\alpha\in\Lambda}\subset V$, such that
$$
\ls v_{\alpha}\rs=e_{\alpha};\,
\ls v-v_{\alpha}\rs=f_{\alpha},\,\alpha\in\Lambda.
$$
Moreover, using the fact that $e_{\alpha}\sqsubseteq e_{\beta}$ we may write
$$
e_{\beta}=e_{\alpha}+(e_{\beta}-e_{\alpha});\,e_{\alpha}\bot(e_{\beta}-e_{\alpha});
$$
$$
\ls v_{\beta}-v_{\alpha}\rs=(e_{\beta}-e_{\alpha})=\ls v_{\beta}\rs-\ls v_{\alpha}\rs\in D.
$$
We must note, if $e\in D$, then $(\pi^{D})^{\bot}Se=0$ for every $S\in\mathcal{U}_{+}(E,F)$. The net $(v_{\alpha})_{\alpha\in\Lambda}$ is said to be {\it cut} for $v\in V$
(with respect to the $D$). For such a net the limit $\bolim_{\alpha}Tv_{\alpha}$ exists. Indeed, for all $\beta\geq\alpha$ we have
$$
\ls Tv_{\beta}-Tv_{\alpha}\rs=\ls T(v_{\beta}-v_{\alpha})+Tv_{\alpha}-Tv_{\alpha}\rs
\leq\Psi\ls v_{\beta}-v_{\alpha}\rs
$$
$$
=\pi^{D}\Psi(\ls v_{\beta}-v_{\alpha}\rs)+(\pi^{D})^{\bot}\Psi(\ls v_{\beta}-v_{\alpha}\rs)=
$$
$$
=\pi^{D}\Psi(\ls v_{\beta}\rs-\ls v_{\alpha}\rs)+
(\pi^{D})^{\bot}\Psi(\ls v_{\beta}\rs-\ls v_{\alpha}\rs)\leq
$$
$$
\leq\pi^{D}\Psi(\ls v\rs-\ls v_{\alpha}\rs)
+(\pi^{D})^{\bot}\Psi(\ls v_{\beta}\rs-\ls v_{\alpha}\rs)=
$$
$$
=\pi^{D}\Psi(\ls v\rs-\ls v_{\alpha}\rs)\downarrow 0.
$$
Observe that by  Lemma~\ref{projection} $(\pi^{D})^{\bot}\Psi(\ls v_{\beta}\rs-\ls v_{\alpha}\rs)=0$ for every $\alpha,\beta\in\Lambda$, $\beta\geq\alpha$. Thus, the net $(Tv_{\alpha})$ is $(bo)$-fundamental and $(bo)$-limit exists by $(bo)$-completeness of $W$. By Lemma~\ref{decompos}, the net $(v_{\alpha})$ is unique.
Hence, the operator $\pi^{D}T:V\to W$ is defined by the formula $\pi^{D}Tv=\bolim_{\alpha}Tv_{\alpha}$, where $v_{\alpha}$ is a cut net for $v$. It is clear, that if $v_{1},v_{2}\in V$, $v_{1}\bot v_{2}$ and $(v_{\alpha}^{1})$, $v_{\alpha}^{2}$ are cut nets for $v_{1}, v_{2}$ then $(v_{\alpha}^{1}+v_{\alpha}^{2})$ is a cut net for $v_{1}+v_{2}$. Consequently, taking into account the definition of the operator $\pi^{D}T$, we have
$$
\pi^{D}T(v_{1}+v_{2})=\pi^{D}Tv_{1}+\pi^{D}Tv_{2};\,v_{1},v_{2}\in V;\,v_{1}\bot v_{2}.
$$
Moreover, the operator $\pi^{D}T$ is a dominated Uryson operator by the following inequalities
$$
\ls \pi^{D}Tv\rs=\olim_{\alpha}\ls Tv_{\alpha}\rs\leq\olim\limits_{\alpha}\Psi(\ls v_{\alpha}\rs)=\pi^{D}\Psi(\ls v\rs).
$$
So, $\pi^{D}T\in\mathcal{D}_{U}(E,F)$ and $\ls\pi^{D}T\rs\leq\pi^{D}\Psi$. For the operator $T-\pi^{D}T$
we may write
$$
\ls (T-\pi^{D}T)v\rs=\olim\limits_{\alpha}\ls Tv-Tv_{\alpha}\rs
\leq\olim\limits_{\alpha}\Psi(\ls v-v_{\alpha}\rs)=(\pi^{D})^{\bot}\Psi.
$$
Therefore $\ls T-\pi^{D}T\rs\leq(\pi^{D})^{\bot}\Psi$. Next, we obtain
$$
\ls T\rs\leq\ls\pi^{D}T\rs+\ls T-\pi^{D}T\rs\leq\pi^{D}\Psi+(\pi^{D})^{\bot}\Psi=\Psi=\ls T\rs;
$$
and therefore $\ls \pi^{D}T\rs=\pi^{D}\Psi$ and $\ls T-\pi^{D}T\rs=(\pi^{D})^{\bot}\Psi$.
Let us prove the uniqueness of the operator $\pi^{D}T$. Assume that there is an operator $\overline{T}$  with the same properties as $\pi^{D}T$:
$$
\ls\overline{T}\rs=\pi^{D}\ls T\rs ;\,
\ls T-\overline{T}\rs=\ls T\rs-\pi^{D}\ls T\rs.
$$
Then we may write
$$
\pi^{D}\ls\overline{T}-\pi^{D}T\rs\leq\pi^{D}\ls T-\pi^{D}T\rs+
$$
$$
+\pi^{D}\ls T-\overline{T}\rs=\pi^{D}(2(\pi^{D})^{\bot}\ls T\rs)=0;
$$
$$
(\pi^{D})^{\bot}\ls \overline{T}-\pi^{D}T\rs\leq(\pi^{D})^{\bot}(\ls\pi^{D}T\rs+\ls\overline{T}\rs)=
$$
$$
=(\pi^{D})^{\bot}(2\pi^{D}\ls T\rs)=0.
$$
Finally we have $\ls \overline{T}-\pi^{D}T\rs=0$ or $\overline{T}=\pi^{D}T$.
\end{proof}

\begin{cor}\label{domspace-3}
Let $T$, $D$ be the same as in Lemma~\ref{domspace-2}. Then the  map $\pi^{D}:T\mapsto\pi^{D}T$
is a linear projection in $\mathcal{D}_{U}(V,W)$.
\end{cor}

\begin{proof}
It is proven that $\pi^{D}$ is a band projection in $\mathcal{U}(E,F)$. Therefore we may write
$$
\ls T-\pi^{D}T\rs=\ls T\rs-\pi^{D}\ls T\rs,
$$
and replacing $T$ with $\pi^{D}T$ we have
$$
\ls \pi^{D}T-(\pi^{D})^{2}T\rs=\ls \pi^{D}T\rs-(\pi^{D})^{2}\ls T\rs=0.
$$
\end{proof}

\begin{lemma}\label{domspace-4}
Let $(V,E)$, $(W,F)$ be the same as in Lemma~\ref{domspace-th}. Suppose $(T_{\alpha})_{\alpha\in\Lambda}$
is a net of dominated Uryson operators, so that for some $R\in\mathcal{D}_{U}(V,W)$
the equality  $\ls R-T_{\alpha}\rs\wedge\ls T_{\alpha}\rs=0$ is valid for all $\alpha\in\Lambda$
and there exists $S:=\olim_{\alpha}\ls T_{\alpha}\rs$. Then $S\in\mathcal{F}_{\ls R\rs}$ and
the equality $T:=\bolim_{\alpha}T_{\alpha}$ well defines a dominated Uryson operator $T:V\to W$ with $\ls T\rs=S$.
\end{lemma}

\begin{proof}
By $T_{\alpha}\bot(R-T_{\alpha})$ in view (\cite{Ku}, $2.1.2$) we deduce
$$
\ls R\rs=\ls R-T_{\alpha}+T_{\alpha}\rs=\ls R-T_{\alpha}\rs+\ls T_{\alpha}\rs
$$
and $\ls R-T_{\alpha}\rs=\ls R\rs-\ls T_{\alpha}\rs$ for every $\alpha\in\Lambda$.
Therefore $\ls T_{\alpha}\rs\in\mathcal{F}_{\ls R\rs}$ and $S\in\mathcal{F}_{\ls R\rs}$.
Denote $\Psi:=\ls R\rs$ for short. Since $\ls T_{\alpha}-T_{\beta}\rs\leq 2\Psi$, and $2S$ is a fragment of $2\Psi$, we may write
\begin{align*}
\ls T_{\alpha}-T_{\beta}\rs=(2\Psi-2S+2S)\wedge\ls T_{\alpha}-T_{\beta}\rs \\
\leq(2\Psi-2S)\wedge\ls T_{\alpha}-T_{\beta}\rs+
2S\wedge\ls T_{\alpha}-T_{\beta}\rs\leq \\
\leq(2\Psi-2S)\wedge(\ls T_{\alpha}\rs+\ls T_{\beta}\rs)+2S\wedge(\ls R-T_{\alpha}\rs+\ls R-T_{\beta}\rs) \\
=(2\Psi-2S)\wedge(\ls T_{\alpha}\rs+\ls T_{\beta}\rs)+2S\wedge(2\ls R\rs-\ls T_{\alpha}\rs-\ls T_{\beta}\rs).
\end{align*}

Thus, we have  that the net $(T_{\alpha})_{\alpha\in\Lambda}$ is $(bo)$-fundamental. Then there
exists an orthogonally additive operator $T=\bolim_{\alpha}T_{\alpha}$. Moreover,we obtain
$$
\ls Tv\rs=\olim\ls T_{\alpha}v\rs\leq\olim_{\alpha}\ls T_{\alpha}\rs(\ls v\rs)\leq S(\ls v\rs).
$$
Therefore $\ls T\rs\leq S$ and $\ls R-T\rs\leq\Psi-S$. Finally, we obtain $\ls T\rs=S$ and $\ls R-T\rs=\Psi-S$.
\end{proof}

\begin{lemma}\label{domspace-5}
Let $(V,E)$, $(W,F)$ be the same as in Lemma~\ref{domspace-th}. Then the dominant norm
$\ls\cdot\rs:\mathcal{D}_{U}(W,W)\to \mathcal{U}^{ev}(E,F)$ is disjointly decomposable.
\end{lemma}
\begin{proof}[Proof of Lemma~\ref{domspace-5}]
Take pairwise disjoint projections $\sigma_{1},\dots,\sigma_{n}\in\mathfrak{B}(F)$ and elements
$e_{1},\dots, e_{n}\in E$. Assign
$$
\rho T=\bigvee\limits_{i=1}^{n}\sigma_{i}\pi^{e_{i}};\,\rho T:=\sigma_{1}\circ(\pi_{1}T)+\dots+\sigma_{n}\circ(\pi_{n}T),
$$
where $\pi_{i}=\pi^{e_{i}}$, $1\leq i\leq n$. Then  by Lemma~\ref{domspace-th} we may write
$$
\ls\rho T\rs=\sigma_{1}\ls\pi_{1}T\rs+\dots+\sigma\ls\pi_{n}T\rs=\sigma_{1}\pi_{1}\ls T\rs+\dots+\sigma\pi_{n}\ls T\rs=\rho\ls T\rs;
$$
$$
\ls\rho^{\bot} T\rs=\sigma_{1}\ls\pi_{1}^{\bot}T\rs+\dots+\sigma\ls\pi_{n}^{\bot}T\rs=\sigma_{1}\pi_{1}^{\bot}\ls T\rs+\dots+\sigma\pi_{n}^{\bot}\ls T\rs=\rho^{\bot}\ls T\rs.
$$
Then take $\rho\in\mathcal{A}(\ls T\rs)^{\uparrow}$. Then there
exists a decreasing net $\rho_{\alpha}$ of projections in $\mathcal{A}(\ls T\rs)$ such that
$\rho=\sup_{\alpha}\rho_{\alpha}$. For each $\alpha$, the operator $\rho_{\alpha}T\in\mathcal{D}_{U}(V,W)$
is well defined, moreover $\ls \rho_{\alpha}T\rs=\rho_{\alpha}\ls T\rs$ and $\ls\rho_{\alpha}^{\bot}T\rs=\rho_{\alpha}^{\bot}\ls T\rs$. By Lemma~\ref{domspace-4}, there exists
a dominated Uryson operator $\rho T=\bolim_{\alpha}\rho_{\alpha}T$, and $\ls\rho T\rs=\rho\ls T\rs$, $\ls\rho^{\bot}T\rs=\rho^{\bot}\ls T\rs$. Using the same arguments, we may establish
the latter equality for the case where $\rho\in\mathcal{A}(\ls T\rs)^{\uparrow\downharpoonleft\uparrow}$.
Thus, for arbitrary fragments $\Psi_{1}=\rho\ls T\rs$ and $\Psi_{2}=\rho^{\bot}\ls T\rs$ of
the dominants norm, we have $T=T_{1}+T_{2}$ and $\ls T_{i}\rs=\Psi_{i}$, $i\in\{1,2\}$,
whenever $T_{1}=\rho T$ and $T_{2}=\rho^{\bot}T$.
\end{proof}

\begin{proof}[Proof of Theorem~\ref{BK}]
Using the fact that every $(bo)$-complete and d-decomposable lattice-normed space is decomposable and applying lemmas \ref{domspace-1} - \ref{domspace-5} we complete the proof.
\end{proof}

\section{Completely additive and laterally continuous orthogonally additive  operator}

In this section we consider completely additive and laterally continuous orthogonally additive operators
and establish some of their properties.

Let $(V,E)$ be a lattice-normed space. A net $(v_{\alpha})_{\alpha\in\Lambda}\subset V$ is said to be {\it laterally } convergent to $v\in V$ if $v=\bolim_{\alpha}v_{\alpha}$ and $\ls v_{\beta}-v_{\gamma}\rs\bot\ls v_{\gamma}\rs$ for all $\beta,\gamma\in\Lambda$, $\beta\geq\gamma$.

\begin{definition}
Let $(V,E)$, $(W,F)$ be lattice normed spaces and $D\subset V$. An orthogonally additive operator $T:V\to W$
is called {\it laterally continuous on $D$}({\it laterally continuous}), if for every  laterally convergent net $(v_{\alpha})\subset D$ ($(v_{\alpha})\subset D$) the net $Tv_{\alpha}$ is laterally convergent. The vector subspace of all laterally continuous dominated Uryson operator is denoted by $\mathcal{D}_{U}^{n}(V,W)$.
\end{definition}
\begin{lemma}\label{lemma-ext}
Let $E,F$ be vector lattices with $F$ Dedekind complete and  $D\subset E$ be an admissible subset of $E$.
If $S\in\mathcal{U}_{+}(E,F)$ is a laterally continuous operator on $D$ then there exists a unique
laterally continuous operator $\overline{S}\in\mathcal{U}_{+}(E,F)$ such that $Se=\overline{S}e$ for every $e\in D$.
\end{lemma}
\begin{proof}
Consider the map
$$
\overline{S}e=\sup\{Se_{0}:\,e_{0}\sqsubseteq e;\,e_{0}\in D\}.
$$
Using the same arguments as in in the proof of Lemma~\ref{le:01} we can prove that $\overline{S}\in\mathcal{U}_{+}(E,F)$.
Let us prove that $\overline{S}$ is a laterally continuous operator. Take a laterally convergent net $(e_{\alpha})_{\alpha\in\Lambda}\subset E$ with $e=\olim_{\alpha}e_{\alpha}$. It is enough
to show that $\overline{S}e\leq\sup_{\alpha}\overline{S}(e_{\alpha})$. If $f\sqsubseteq e$ and $f\in D$, by the Riesz decomposition property, there exist a laterally convergent net $f_{\alpha}\subset D$ such that $f=\olim_{\alpha}f_{\alpha}$ and $f_{\alpha}\sqsubseteq e_{\alpha}$ for every $\alpha\in\Lambda$. Then we have
$$
Sf=\olim_{\alpha}Tf_{\alpha}=\sup_{\alpha}Tf_{\alpha}\leq\sup_{}\overline{S}e_{\alpha}.
$$
Passing to the supremum over all fragments $f\sqsubseteq e,\,f\in D$ we may write $\overline{S}e\leq\sup_{\alpha}\overline{S}e_{\alpha}$.
\end{proof}

\begin{thm}\label{thm-lat-cont}
Let $(V,E)$ be a lattice-normed space and let $(W,F)$ be a Banach-Kantorovich space.
Then a dominated Uryson operator $T:V\to W$ is laterally continuous if and only if its exact dominant $\ls T\rs:E\to F$ is.
$$
T\in\mathcal{D}_{U}^{n}(V,W)\Longleftrightarrow\ls T\rs\in\mathcal{U}_{+}^{n}(E,F).
$$
\end{thm}

\begin{proof}
Let $\ls T\rs$ be a laterally continuous operator. Take a laterally convergent net $(v_{\alpha})\subset V$ with  $v=\bolim_{\alpha}v_{\alpha}$. Then we have
$$
\ls Tv-Tv_{\alpha}\rs=\ls T(v-v_{\alpha}+v_{\alpha})-Tv_{\alpha}\rs=
$$
$$
\ls T(v-v_{\alpha})\rs\leq\ls T\rs(\ls v-v_{\alpha}\rs)\downarrow 0.
$$
Therefore $T$ is laterally continuous. Let us prove a converse   assertion.
Suppose $T\in\mathcal{D}_{U}^{n}(V,W)$. Take a $e\in\widetilde{E}_{+}$
and a laterally convergent net $(e_{\alpha})_{\alpha\in\Lambda}\subset\widetilde{E}_{+}$, so
that $e=\olim_{\alpha}e_{\alpha}$. Assign
$$
g=\sup_{\alpha}\sup\Big\{\sum\limits_{i=1}^{n}\ls Tv_{i}\rs:\,v_{1},\dots,v_{n}\in V;\,\bigsqcup\limits_{i=1}^{n}\ls v_{i}\rs=e_{\alpha},\,n\in\Bbb{N}\}.
$$
Then we have $g=\sup_{\alpha}\ls T\rs(e_{\alpha})\leq\ls T\rs(e)$. Let us show that $\ls T\rs(e)\leq\sup_{\alpha}\ls T\rs(e_{\alpha})$. Consider a finite family of mutually disjoint elements $v_{1},\dots, v_{n}$ of $V$ with the property $\bigsqcup\limits_{i=1}^{n}\ls v_{i}\rs=e$. Given $\alpha\in\Lambda$, we  associate with each $i\in\{1,\dots, n\}$ a representation $v_{i}=u_{i,\alpha}+w_{i,\alpha}$, $u_{i,\alpha}\bot w_{i,\alpha}$ for every $i\in\{1,\dots,n\},\alpha\in\Lambda$, so that
$$
\ls v_{i}\rs=\ls u_{i,\alpha}\rs+\ls w_{i,\alpha}\rs;\,
\bigsqcup\limits_{i=1}^{n}\ls u_{i,\alpha}\rs=e_{\alpha};
\bigsqcup\limits_{i=1}^{n}\ls w_{i,\alpha}\rs=e-e_{\alpha}.
$$
Since $(e_{\alpha})$ laterally converges to $e$, we have $\ls v_{i}-u_{i,\alpha}\rs=\ls w_{i,\alpha}\rs$ and therefore, $u_{i,\alpha}$ laterally converges to $v_{i}$ for every $i\in\{1,\dots, n\}$. Then we have
$$
\sum\limits_{i=1}^{n}\ls Tv_{i}\rs=\olim_{\alpha}(\sum\limits_{i=1}^{n}\ls Tu_{i,\alpha}\rs).
$$
On the other hand for any $\beta\in\Lambda$ we have
$$
\sum\limits_{i=1}^{n}\ls Tu_{i,\beta}\rs\leq\Big\{\sum\limits_{i=1}^{n}\ls Tu_{i}\rs:\,
u_{1},\dots, u_{n}\subset V;\,
\bigsqcup\limits_{i=1}^{n}\ls u_{i}\rs=e_{\beta};\,n\in\Bbb{N}\Big\}=
$$
$$
=\ls T\rs(e_{\beta})\leq\sup_{\alpha}\ls T\rs(e_{\alpha})=g.
$$
Passing to the $o$-limit over $\beta$ in the latter inequalities, we obtain $\sum\limits_{i=1}^{n}\ls Tu_{i}\rs\leq g$. Finally, taking the supremum over all mutually disjoint $\{u_{1},\dots, u_{n}\}$ we have $\ls T\rs(e)\leq g$. Thus, we have proved that the operator $\ls T\rs$ is laterally continuous on the admissible set $\widetilde{E}_{+}$. By Lemma~\ref{lemma-ext}, we obtain that $\ls T\rs$ coincides with $g$.
\end{proof}
Let $\Lambda$ be an index set and $\Theta$ be the set of all finite subset of $\Lambda$. Recall that a family $(v_{\alpha})_{\alpha\in\Lambda}$ of the elements of a lattice-normed space $V$ is called {\it $(bo)$-summable} if the
net $(w_{\theta})_{\theta\in\Theta}$, $w_{\theta}=\sum\limits_{\alpha\in\theta}v_{\alpha}$, is $(bo)$-convergent and $w=\bolim_{\theta}w_{\theta}$. We denote the element $w$ by $\bosum_{\alpha}v_{\alpha}$.
\begin{definition}\label{def-compad}
Let $(V,E)$, $(W,F)$ be lattice normed spaces. An orthogonally additive
operator $T:V\to W$ is said to be {\it completely additive} if, for every $(bo)$-summable family of mutually disjoint elements $(v_{\alpha})$ the family $Tv_{\alpha}$ is $(bo)$-summable and
$$
T\Big(\bosum\limits_{\alpha\in\Lambda}v_{\alpha}\Big)=\bosum\limits_{\alpha\in\Lambda}Tv_{\alpha}.
$$
\end{definition}
\begin{thm}
Let $(V,E)$ be a lattice-normed space and let $(W,F)$ be a Banach-Kantorovich space.
Then a dominated Uryson operator $T:V\to W$ is completely additive  if and only if its exact dominant $\ls T\rs:E\to F$ is.
\end{thm}
\begin{proof}
Assume that $\ls T\rs$ is a completely additive, $(v_{\alpha})_{\alpha\in\Lambda}$ is a $(bo)$-summable family, $\Theta$ is the set of all finite subset of $\Lambda$, $\theta\in\Theta$ and  $w_{\theta}=\sum\limits_{\alpha\in\theta}v_{\alpha}$. At first we have
$$
\ls\bosum_{\alpha}v_{\alpha}-w_{\theta}\rs=\ls\bosum_{\alpha}v_{\alpha}\rs-\ls w_{\theta}\rs=
$$
$$
\ls\bosum_{\alpha}v_{\alpha}\rs-\ls\sum\limits_{\alpha\in\theta}v_{\alpha}\rs=
\ls\bosum_{\alpha}v_{\alpha}\rs-\sum\limits_{\alpha\in\theta}\ls v_{\alpha}\rs\downarrow 0;
$$
$$
\bosum_{\alpha}\ls v_{\alpha}\rs-\sum\limits_{\alpha\in\theta}\ls v_{\alpha}\rs\downarrow 0.
$$
Therefore $\ls\bosum_{\alpha}v_{\alpha}\rs=\bosum_{\alpha}\ls v_{\alpha}\rs$.
Now we may write
$$
w=\bosum_{\alpha}v_{\alpha}=\bosum\limits_{\alpha\notin\theta}v_{\alpha}+w_{\theta};\,
\Big(\bosum\limits_{\alpha\notin\theta}v_{\alpha}\Big)\bot w_{\theta};
$$
$$
\ls Tw-Tw_{\theta}\rs=\ls T\Big(\bosum\limits_{\alpha\notin\theta}v_{\alpha}\Big)\rs\leq\ls T\rs\ls\bosum\limits_{\alpha\notin\theta}v_{\alpha}\rs
$$
$$
=\ls T\rs\Big(\bosum\limits_{\alpha\notin\theta}\ls v_{\alpha}\rs\Big)\downarrow 0.
$$
Hence, the operator $T$ is completely additive.

Now assume that $T$ is completely additive.
Using the fact that $\ls T\rs$ is an even operator, we may consider only
$(bo)$-summable families $(e_{\alpha})_{\alpha\in\Lambda}$, where $e_{\alpha}\in E_{+}$ for every $\alpha\in\Lambda$.
Take $v\in V$ and a finite family $\rho_{1},\dots,\rho_{n}$ of mutually disjoint order projections in $V$ such
that $\bigvee\limits_{i=1}^{n}\rho_{i}=\text{Id}_{V}$.
Assign $e:=\ls v\rs=\bosum\limits_{\alpha}e_{\alpha}$,
$e_{\alpha}\bot e_{\beta}$, $\alpha\neq\beta$, $e_{\alpha}\in E_{+}$ and $\sigma_{\alpha}$ be the projection onto the band $\{e_{\alpha}\}$. Then we have
$$
\sum\limits_{i=1}^{n}\ls T(\rho_{i}v)\rs=\sum\limits_{i=1}^{n}\ls T\Big(\sum\limits_{\alpha\in\Lambda}\sigma_{\alpha}\rho_{i}v\Big)\rs=
\sum\limits_{i=1}^{n}\ls\Big(\sum\limits_{\alpha\in\Lambda}T\sigma_{\alpha}\rho_{i}v\Big)\rs\leq
$$
$$
\leq\sum\limits_{\alpha\in\Lambda}\sum\limits_{i=1}^{n}\ls T(\rho_{i}\sigma_{\alpha}v)\rs\leq
\sum\limits_{\alpha\in\Lambda}\ls T\rs(\sigma_{\alpha}\ls v\rs).
$$
Passing to the supremum over all finite families $\rho_{1}v,\dots,\rho_{n}v$,
where $\rho_{i}\bot\rho_{j}$, $i\neq j$, $\bigvee\limits_{i=1}^{n}\rho_{i}=\text{Id}_{V}$,
we obtain
$$
\ls T\rs(e)\leq\sum\limits_{\alpha\in\Lambda}\ls T\rs(e_{\alpha}).
$$
The reverse inequality is straightforward, we prove the complete additivity of the operator
$T$ on an admissible set $\widetilde{E}_{+}$. If $e,e_{\alpha}\in E_{+}$
then for arbitrary $e'\in\widetilde{E}_{+}$, $e'\sqsubseteq e$, taking into account what we have proven,
we may write
$$
\ls T\rs(e')=\sum\limits_{\alpha\in\Lambda}\ls T\rs(\sigma_{\alpha}e')\leq\sum\limits_{\alpha\in\Lambda}\ls T\rs(e_{\alpha}).
$$
Passing to the supremum over all $e'\in\widetilde{E}_{+}$, $e'\sqsubseteq e$ we have $\ls T\rs(e)\leq\sum\limits_{\alpha\in\Lambda}\ls T\rs(e_{\alpha})$.
\end{proof}


\begin{thebibliography}{8}

\bibitem{Ab}~Abramovich\,Y.~A., Aliprantis\,C.~D. An Invitation to
Operator Theory.---AMS, 2002.


\bibitem{Al} \textsc{C.~D.~Aliprantis}, \textsc{O.~Burkinshaw}, \emph{Positive Operators}, Springer, Dordrecht. (2006).

\bibitem{Al-1}\textsc{C.~D.~Aliprantis}, \textsc{O.~Burkinshaw}, \emph{The components of a positive operator}, Math. Z., 184(2) (1983), pp.~245-257.

\bibitem{Ben}\textsc{M.~A.~Ben Amor}, \textsc{M.~Pliev}, \emph{Decomposition of an abstarct Uryson operator}, Preprint.


\bibitem{Ku} \textsc{A.~G.~Kusraev}, \emph{Dominated Operators}, Kluwer Acad. Publ., Dordrecht--Boston--London (2000).



\bibitem{Ku-1} \textsc{A.~G.~Êusraev, M.~A.~Pliev}, \emph{Orthogonally additive operators on lattice-normed spaces}, Vladikavkaz Math. J. No~3 (1999), pp.~33-43.

\bibitem{Ku-2} \textsc{A.~G.~Êusraev, M.~A.~Pliev}, \emph{Weak integral representation of the dominated orthogonally additive operators}, Vladikavkaz Math. J. No~4 (1999), pp.~22-39.

\bibitem{KS} \textsc{A.~G.~Êusraev, M.~A.~Strizhevski}, \emph{Lattice-normed spaces and dominated operators}, Studies on Geometry and Functional Analysis. Vol.7. Trudy Inst. Mat.(Novosibirsk), Novosibirsk, 1987, pp.~132-158.

\bibitem{Maz-1} \textsc{J.~M.~Maz\'{o}n, S.~Segura de Le\'{o}n}, \emph{Order bounded ortogonally additive operators}, Rev. Roumane Math. Pures Appl. 35, No~4 (1990), pp.~329-353.

\bibitem{Maz-2} \textsc{J.~M.~Maz\'{o}n, S.~Segura de Le\'{o}n}, \emph{Uryson operators}, Rev. Roumane Math. Pures Appl. 35, No~5 (1990), pp.~431-449.



\bibitem{Pag}\textsc{Pagter, de B}, \emph{The components of a positive operator}, Indag. Math. 48(2) (1983), pp.~229-241.


\bibitem{Pl-3} \textsc{M.~Pliev}, \emph{Uryson operators on the spaces with mixed norm}, Vladikavkaz Math. J. No~3 (2007), pp~47-57.

\bibitem{Pl-4} \textsc{M.~Pliev}, \emph{Order projections in the space of Uryson operators}, Vladikavkaz Math. J. No~4 (2006), pp.~38-44.

\bibitem{Pl-1}~Pliev\,M. \emph{Projection of positive Uryson operator},  Vladikavkaz Math. J. No~4, 2005.  pp.~45-51.


\bibitem{Seg} \textsc{S.~Segura de Leon}, \emph{Bukhvalov type characterization of Urysohn operators}, Studia Math.
99, No~3 (1991), pp.~199-220.


\bibitem{Za} \textsc{A.~G.~Zaanen}, \emph{Riesz spaces II}, North Holland, Amsterdam, (1983).


\end{thebibliography}
\end{document}